\documentclass[11pt]{article}
\usepackage{amsmath, amssymb,amscd,verbatim,appendix,lpic,multicol}
\usepackage[mathscr]{eucal}
\usepackage{amscd}
\usepackage{stmaryrd}
\usepackage{enumerate}
\usepackage{url}
\usepackage{macros}
\usepackage{thms}

\usepackage{wasysym}

\usepackage[margin=1in]{geometry}

\AtBeginDocument{%
   \def\MR#1{}
}

\title{
A remark on strict independence relations
}

\date{November 7, 2015; updated: March 2, 2016}

\author{
Gabriel Conant\\
University of Notre Dame\\
gconant@nd.edu
}

\begin{document}

\maketitle

\begin{abstract}\noindent
We prove that if $T$ is a complete theory with weak elimination of imaginaries, then there is an explicit bijection between strict independence relations for $T$ and strict independence relations for $T^{\eq}$. We use this observation to show that if $T$ is the theory of the \Fraisse\ limit of finite metric spaces with integer distances, then $T^{\eq}$ has more than one strict independence relation. This answers a question of Adler \cite[Question 1.7]{Adgeo}.
\end{abstract}

\section{Introduction}

Let $T$ be a complete first-order theory with a sufficiently saturated monster model $\M$. We use $A,B,C,\ldots$ to denote small subsets of $\M$, where $A$ is \emph{small} (written $A\subset\M$) if $\M$ is $|T(A)|^+$-saturated. We take the following definition from Adler \cite{Adgeo}.

\begin{definition}\label{def:SIR}
Let $\ind$ be a ternary relation on small subsets of $\M$. We define the following axioms.
\begin{enumerate}[$(i)$]
\item (\emph{invariance}) If $A\ind_C B$ and $\sigma\in\Aut(\M)$, then $\sigma(A)\ind_{\sigma(C)}\sigma(B)$. 
\item (\emph{monotonicity}) If $A\ind_C B$, $A'\seq A$, and $B'\seq B$, then $A'\ind_C B'$.
\item (\emph{base monotonicity}) If $A\ind_D B$ and $D\seq C\seq B$ then $A\ind_C B$.
\item (\emph{transitivity}) If $D\seq C\seq B$, $A\ind_D C$ and $A\ind_C B$ then $A\ind_D B$.
\item (\emph{extension}) If $A\ind_C B$ and $\hat{B}\supseteq B$ then there is $A'\equiv_{BC} A$ such that $A\ind_C \hat{B}$.
\item (\emph{finite character}) If $A_0\ind_C B$ for all finite subsets $A_0\seq A$, then $A\ind_C B$.
\item (\emph{local character}) For every $A$, there is a cardinal $\kappa(A)$ such that, for every $B$, there is $C\seq B$ with $|C|<\kappa(A)$ and $A\ind_C B$.
\item (\emph{symmetry}) If $A\ind_C B$ then $B\ind_C A$.
\item (\emph{anti-reflexivity}) $a\ind_C a$ implies $a\in\acl(C)$. 
\end{enumerate}
A ternary relation $\ind$ is a \textbf{strict independence relation} for $T$ if it satisfies $(i)$ through $(ix)$.\footnote{In \cite{Adgeo}, Adler formulates transitivity on the left, and proves symmetry as a consequence of the rest of the axioms. We are using the more standard formulation of transitivity on the right, and therefore include symmetry. Adler also includes an axiom called \emph{normality}, which is a consequence of invariance, extension, and symmetry \cite[Remark 1.2(1)]{Adgeo}.} $T$ is \textbf{rosy} if there is a strict independence relation for $T^{\eq}$.
\end{definition}
 
Rosy theories were developed by Ealy and Onshuus as a way to define a class of theories equipped with the weakest possible notion of independence, which still satisfies basic desirable properties (see, e.g., \cite{EaOn}). Specifically, if $T$ is rosy, then there is a distinguished strict independence relation for $T^{\eq}$, called \emph{thorn-forking independence}, which is canonical in the sense that it is the weakest such relation \cite{Adgeo}. This fact motivates Question 1.7 of \cite{Adgeo}, which asks whether there is a theory $T$ such that $T^{\eq}$ has more than one strict independence relation. Rephrased in the negative, the question asks if the axioms of strict independence relations \emph{characterize} thorn-forking in $T^{\eq}$ for rosy theories (analogous to the characterizations of forking in stable and simple theories due, respectively, to Harnik-Harrington \cite{HaHa} and Kim-Pillay \cite{KiPi}).

In this note, we answer Adler's question by exhibiting two distinct strict independence relations for $T^{\eq}$, where $T$ is the theory of the \emph{integral Urysohn space} (i.e. the \Fraisse\ limit of metric spaces with integer distances). To accomplish this, we first prove a general result that if $T$ is a complete theory with weak elimination of imaginaries, then there is an explicit bijection between strict independence relations for $T$ and strict independence relations for $T^{\eq}$. We then show that if $T$ is the theory of the integral Urysohn space (shown in \cite{CoDM2} to have weak elimination of imaginaries), then $T$ has more than one strict independence relation.  We conclude with further remarks on the motivation for Adler's question, as well as a generalization of our results to a certain family of theories (including that of the \emph{rational} Urysohn space).

\section{Imagining strict independence}

Let $T$ be a complete first-order theory with monster model $\M$. 

\begin{definition}\label{def:main}$~$
\begin{enumerate}
\item Fix $e\in\M^{\eq}$. A finite tuple $c\in\M$ is a \textbf{weak canonical parameter} for $e$ if $c\in\eacl(e)$ and $e\in\edcl(c)$. \item $T$ has \textbf{weak elimination of imaginaries} if every $e\in\M^{\eq}$ has a weak canonical parameter in $\M$.
\item Suppose $A\subset\M^{\eq}$. We say a subset $A_*\subset\M$ is a \textbf{weak code for $A$} if $A_*=\bigcup_{e\in A}c(e)$, where $c(e)$ is a weak canonical parameter for $e$.
\end{enumerate}
\end{definition}

Note that if $T$ has weak elimination of imaginaries then any subset of $\M^{\eq}$ has a weak code in $\M$. Therefore, if $\ind$ is a ternary relation on $\M$, we can use weak codes to define a ternary relation on $\M^{\eq}$ in the following way.

\begin{definition}
Suppose $T$ has weak elimination of imaginaries, and $\ind$ is a ternary relation on $\M$. Define the ternary relation $\ind^{\eq}$ on $\M^{\eq}$ such that, given $A,B,C\subset\M^{\eq}$,
$$
A\textstyle\ind^{\eq}_C B\miff A_*\ind_{C_*}B_*\text{ for some weak codes $A_*,B_*,C_*$ for $A,B,C$.}
$$
\end{definition}

The following basic exercise in forking calculus will allow us to replace ``for some weak codes" in the last definition with ``for all weak codes". 

\begin{exercise}\label{Adfact}
If $\ind$ is a strict independence relation for $T$, and $A,B,C\subset\M$, then $A\textstyle \ind_C B$ if and only if $\acl(A)\ind_{\acl(C)}\acl(B)$.
\end{exercise}

\begin{lemma}\label{lem:acl}
Assume $T$ has weak elimination of imaginaries, and suppose $\ind$ is a strict independence relation for $T$. Then, for any $A,B,C\subset\M^{\eq}$, 
$$
A\textstyle\ind^{\eq}_C B\miff A_*\ind_{C_*}B_*\text{ for all weak codes $A_*,B_*,C_*$ for $A,B,C$.}
$$
\end{lemma}
\begin{proof}
First, note that if $A\subset\M^{\eq}$ and $A_*$ is a weak code for $A$, then we have $\eacl(A_*)=\eacl(A)$, and so $\acl(A_*)=\eacl(A)\cap\M$. In particular, $\acl(A_*)$ does not depend on the choice of weak code. With this observation in hand, we prove the claim. Since any subset of $\M^{\eq}$ has a weak code in $\M$, the right-to-left implication is trivial. For left-to-right, suppose $A\ind^{\eq}_C B$ and $A_*,B_*,C_*$ are weak codes for $A,B,C$. By definition, there are weak codes $A_{**},B_{**},C_{**}$ for $A,B,C$ such that $A_{**}\ind_{C_{**}}B_{**}$. By Exercise \ref{Adfact} and the above observation,
$$
\textstyle A_{**}\ind_{C_{**}}B_{**}\Rightarrow \acl(A_{**})\ind_{\acl(C_{**})}\acl(B_{**})
\textstyle\Rightarrow \acl(A_*)\ind_{\acl(C_*)}\acl(B_*)\Rightarrow A_*\ind_{C_*}B_*,
$$
as desired.
\end{proof}

\begin{theorem}\label{thm:transfer}
Suppose $T$ has weak elimination of imaginaries. Then the map $\ind\mapsto \ind^{\eq}$ is a bijection from strict independence relations for $T$ to strict independence relations for $T^{\eq}$.
\end{theorem}
\begin{proof}
We first show that the map is well-defined, i.e., if $\ind$ is a strict independence relation for $T$ then, in $\M^{\eq}$, $\ind^{\eq}$ satisfies the axioms of Definition \ref{def:SIR}. Invariance, monotonicity, base monotonicity, symmetry, transitivity, and finite character, are all straightforward, but rely on Lemma \ref{lem:acl}. To illustrate this, we show transitivity.

\emph{Transitivity}: Suppose $A,B,C,D\subset\M^{\eq}$, with $D\seq C\seq B$, $A\ind^{\eq}_D C$, and $A\ind^{\eq}_C B$. We want to show $A\ind^{\eq}_D B$. We may find weak codes $A_*,B_*,C_*,D_*$ for $A,B,C,D$ such that $D_*\seq C_*\seq B_*$. By Lemma \ref{lem:acl}, we have $A_*\ind_{D_*}C_*$ and $A_*\ind_{C_*}B_*$. By transitivity for $\ind$, we have $A_*\ind_{D_*}B_*$, and so $A\ind^{\eq}_D B$.

The rest of the axioms (\emph{anti-reflexivity}, \emph{local character}, and \emph{extension}) require more than just Lemma \ref{lem:acl}.

\emph{Anti-reflexivity}: Fix $a\in\M^{\eq}$ and $C\subset\M^{\eq}$, with $a\ind^{\eq}_C a$. Let $a_*$ and $C_*$ be weak codes for $a$ and $C$ (in particular, $a_*$ is a weak canonical parameter for $a$). We have $a_*\ind_{C_*} a_*$ by Lemma \ref{lem:acl}, and so $a_*\in\acl(C_*)$ by anti-reflexivity for $\ind$. Then $a\in\eacl(a_*)\seq\eacl(C_*)=\eacl(C)$.

\emph{Local character}: Fix $A\subset\M^{\eq}$, and let $A_*$ be a weak code for $A$. By local character for $\ind$, there is a $\kappa$ such that, for all $B_0\subset\M$, there is $C_0\seq B_0$, with $|C_0|<\kappa$ and $A_*\ind_{C_0}B_0$. Suppose $B\subset\M^{\eq}$, and let $B_*$ be a weak code for $B$. Fix $C_0\seq B_*$ such that $|C_0|<\kappa$ and $A_*\ind_{C_0}B_*$. By assumption, we may write $B_*=\bigcup_{e\in B}c(e)$, where $c(e)$ is a fixed weak canonical parameter for $e$. Given $d\in C_0$, we fix $e_d\in B$ such that $d\in c(e_d)$. Let $C=\{e_d:d\in C_0\}$ and $C_*=\bigcup_{d\in C_0}c(e_d)$. Then $|C|<\kappa$, $C_*$ is a weak code for $C$, and $C_0\seq C_*\seq B_*$. By base monotonicity for $\ind$, we have $A_*\ind_{C_*}B_*$, and so $A\ind^{\eq}_C B$.

\emph{Extension}: Fix $A,B,C\subset\M^{\eq}$, with $A\ind^{\eq}_C B$, and suppose $B\seq\hat{B}$. Let $A_*,B_*,C_*$ be weak codes for $A,B,C$, with $A_*\ind_{C_*}B_*$. Then we may find a weak code $\hat{B}_*$ for $\hat{B}$, with $B_*\seq\hat{B_*}$. By extension for $\ind$, there is $A'_*\equiv_{B_*C_*}A_*$ such that $A'_*\ind_{C_*}\hat{B}_*$. Fix $\sigma\in\Aut(\M/B_*C_*)$ such that $\sigma(A_*)=A'_*$. Let $A'=\sigma(A)$. Then $A'_*$ is a weak code for $A'$, and so we have $A'\ind^{\eq}_C \hat{B}$. It remains to show that $A'\equiv_{BC} A$ and, for this, it suffices to see that $\sigma$ fixes $BC$ pointwise. Let $B_*=\bigcup_{e\in B}c(e)$. Then $\sigma(c(e))=c(e)$ for all $e\in B$ and so, since $e\in\edcl(c(e))$, it follows that $\sigma(e)=e$. Similarly, $\sigma$ also fixes $C$ pointwise.

Finally, we must show that $\ind\mapsto\ind^{\eq}$ is a bijection. For injectivity, simply use the fact any subset $A\subset\M$ is a weak code for itself. For surjectivity, suppose $\ind$ is a strict independence relation for $T^{\eq}$. Let $\ind^*$ be the restriction to subsets of $\M$. Then one easily checks that $\ind^*$ is a strict independence relation for $T$. Using Exercise \ref{Adfact}, we have $\ind^{*,\eq}=\ind$.
\end{proof}

Recall that $T$ is \emph{real rosy} if there is a strict independence relation for $T$. The previous theorem immediately implies the following well-known fact.

\begin{corollary}\label{cor1}
If $T$ is real rosy with weak elimination of imaginaries, then $T$ is rosy.
\end{corollary}

This result is shown explicitly by Ealy and Goldbring in \cite{EaGo} (in the context of continuous logic). The proof of Theorem \ref{thm:transfer} is similar to their work. 

Define \emph{algebraic independence}, denoted $\ind^a$, in $\M$ by: $A\ind^a_C B$ if and only if $\acl(AC)\cap\acl(BC)=\acl(C)$. For the rest of the paper, our focus will be on theories in which algebraic independence coincides with thorn-forking independence, denoted $\ind^{\thrn}$. Therefore, we omit the definition of thorn-forking and refer the reader to \cite{Adgeo}. The next fact is a standard result.

\begin{fact}\label{fact:alg}
The following are equivalent.
\begin{enumerate}[$(i)$]
\item $\ind^a$ coincides with $\ind^{\thrn}$ in $\M$ (resp. in $\M^{\eq}$).
\item $\ind^a$ satisfies base monotonicity in $\M$ (resp. in $\M^{\eq}$).
\item $\ind^a$ is a strict independence relation for $T$ (resp. for $T^{\eq}$).
\end{enumerate}
\end{fact}
\begin{proof}
See Lemma 4.2, Proposition 1.5, and Theorem 4.3 of \cite{Adgeo} for, respectively, $(i)\Rightarrow(ii)$, $(ii)\Rightarrow(iii)$, and $(iii)\Rightarrow(i)$.\end{proof}

\begin{corollary}\label{cor2}
Suppose $T$ has weak elimination of imaginaries. If $A,B,C\subset\M$ then $A\ind^a_C B$ in $\M$ if and only if $A\ind^a_C B$ in $\M^{\eq}$. Moreover, if $\ind^\thrn=\ind^a$ in $\M$ then $\ind^\thrn=\ind^a$ in $\M^{\eq}$.
\end{corollary}
\begin{proof}
Let $\ind^{a,\eq}$ denote the ternary relation in $\M^{\eq}$ obtained by applying the $\eq$-map to $\ind^a$ in $\M$. Recall that, for any $A\subset\M^{\eq}$ and weak code $A_*\subset\M$, we have $\acl(A_*)=\eacl(A)\cap\M$. Using this, it is routine to show that $\ind^{a,\eq}$ coincides with $\ind^a$ in $\M^{\eq}$. The first claim then follows from Lemma \ref{lem:acl}, and the fact that any subset of $\M$ is a weak code for itself. For the second claim, combine Theorem \ref{thm:transfer} and Fact \ref{fact:alg}.
\end{proof}

It is worth noting that Theorem \ref{thm:transfer} becomes false if the assumption of weak elimination of imaginaries is removed. In particular, if $T$ is stable then forking and thorn-forking coincide in $\M^{\eq}$ \cite{EaOn}, and so thorn-forking is the unique strict independence relation for $T^{\eq}$ \cite{Adgeo}. However, there are stable theories (failing weak elimination of imaginaries) for which $T$ has more than one strict independence relation. In fact, for any cardinal $\kappa$, if $T$ is the model completion of the theory of $\kappa$ many equivalence relations, then $T$ is stable and has at least $2^\kappa$ distinct strict independence relations (see \cite[Example 1.5]{Adgeo}).

\section{Integer distance metric spaces}

Let $\cN$ denote the ordered monoid $(\N,+,\leq,0)$. Then the class $\cK_\cN$ of finite metric spaces, with integer distances, is a \Fraisse\ class in the relational language $\cL_\N=\{d_n(x,y):n\in\N\}$, where $d_n(x,y)$ is interpreted as $d(x,y)\leq n$. In particular, $\cK_\cN$ is closed under free amalgamation of metric spaces. Precisely, given integer distance metric spaces $A,B,C$, with $\emptyset\neq C\seq A\cap B$, the \emph{free amalgamation of $A$ and $B$ over $C$} is defined by setting, for $a\in A$ and $b\in B$, $d(a,b)=d_{\max}(a,b/C):=\inf_{c\in C}(d(a,c)+d(c,b))$. 

Let $\cU_\cN$ denote the \Fraisse\ limit of $\cK_\cN$, which we refer to as the \emph{integral Urysohn space}. Then $\cU_\cN$ is the unique (up to isometry) countable, universal, and ultrahomogeneous metric space with integer distances. Let $T_\cN=\Th(\cU_\cN)$, and let $\U_{\cN}$ be a sufficiently saturated monster model of $T_\cN$.

Note that $\U_\cN$ cannot be interpreted as an integer-valued metric space in a way coherent with $T_\cN$. In particular, the type $\{\neg d_n(x,y):n\in\N\}$ is consistent with $T_\cN$, and therefore realized in $\U_\cN$ by points of ``infinite distance". However, this is the only obstruction to viewing $\U_\cN$ as a metric space, and we resolve the issue as follows. Let $\cN^*=(\N\cup\{\infty\},+,\leq,0)$ be an ordered monoid extension of $\cN$, where $\infty+\infty=\infty$ and, for all $n\in\N$, $n<\infty$ and $n+\infty=\infty=\infty+n$. Then $\U_\cN$ can be viewed as an $\cN^*$-valued metric space. We use $d$ to refer to the $\cN^*$-metric on $\U_\cN$. Given $C\subset\U_\cN$ and $a\in\U_\cN$, let $d(a,C)=\inf\{d(a,c):c\in C\}$. Then $C=\emptyset$ implies $d(a,C)=\infty$, and if $C$ is nonempty then there is some $c\in C$ such that $d(a,C)=d(a,c)$. In particular, $d(a,C)=0$ if and only if $a\in C$. For the subsequent work, we will need the following facts about $T_\cN$. 

\begin{fact}\label{thm:UN}
$~$
\begin{enumerate}[$(a)$]
\item $T_\cN$ has quantifier elimination. Consequently, $\acl(C)=C$ for all $C\subset\U_\cN$. Moreover, $\U_\cN$ is a $\kappa^+$-universal and $\kappa$-homogeneous $\cN^*$-metric space, where $\kappa$ is the saturation cardinal of $\U_\cN$.
\item $T_\cN$ has weak elimination of imaginaries.
\end{enumerate}
\end{fact}

Details on these results can be found in \cite{CoDM2}. Our goal is to define two distinct strict independence relations on $T_\cN$. By weak elimination of imaginaries and Theorem \ref{thm:transfer}, we will then obtain two distinct strict independence relations on $T^{\eq}_\cN$. The first strict independence relation is given to us by thorn-forking. 

\begin{theorem}\label{thm:UNrosy}
Thorn-forking is a strict independence relation for $T^{\eq}_\cN$ (i.e. $T_\cN$ is rosy). In particular, $\ind^\thrn$ coincides with $\ind^a$ in $\U_\cN^{\eq}$.
\end{theorem}
\begin{proof}
By Fact \ref{thm:UN}$(a)$, $\ind^a$ satisfies base monotonicity in $\U_\cN$. So the result follows from Fact \ref{fact:alg}, Corollary \ref{cor2}, and Fact \ref{thm:UN}$(b)$.
\end{proof}

We now define what we will show to be a second strict independence relation for $T_\cN$.

\begin{definition}\label{def:II}
Define $\ind^\infty$ on $T_\cN$ such that, given $A,B,C\subset\U_\cN$, $A\ind^\infty_C B$ if and only if, for all $a\in A$, if $d(a,B)=0$ then $d(a,C)=0$, and if $d(a,C)=\infty$ then $d(a,B)=\infty$.
\end{definition}

Since $\acl(C)=C$ for all $C\subset\U_\cN$, and $d(a,C)=0$ if and only if $a\in C$, we can rephrase $\ind^\infty$ as: $A\ind^\infty_C B$ if and only if $A\ind^a_C B$ and, for all $a\in A$, if $d(a,C)=\infty$ then $d(a,B)=\infty$. 

\begin{theorem}\label{thm:inftySIR}
The relation $\ind^\infty$ is a strict independence relation for $T_\cN$.
\end{theorem}
\begin{proof}
Invariance and finite character are easy.

\emph{Symmetry}: Suppose $A\ind^\infty_C B$. Then $B\ind^a_C A$, and so it remains to fix $b\in B$, with $d(b,C)=\infty$, and show $d(a,b)=\infty$ for all $a\in A$. Given $a\in A$, we have $\infty=d(b,C)\leq d(a,b)+d(a,C)$ by the triangle inequality. Therefore, we either directly have $d(a,b)=\infty$, or we have $d(a,C)=\infty$, in which case $A\ind^\infty_C B$ implies $d(a,b)=\infty$.

\emph{Transitivity}, \emph{monotonicity}, and \emph{base monotonicity}: Since we have shown symmetry, we can verify all three axioms with the following claim: if $D\seq C\seq B$ then $A\ind^\infty_D B$ if and only if $A\ind^\infty_D C$ and $A\ind^\infty_C B$. This is clearly true for $\ind^a$, so it suffices to fix $a\in A$ and show that the following are equivalent: 
\begin{enumerate}[$(i)$]
\item $d(a,D)=\infty$ implies $d(a,B)=\infty$;
\item $d(a,D)=\infty$ implies $d(a,C)=\infty$, and $d(a,C)=\infty$ implies $d(a,B)=\infty$.
\end{enumerate}
The implication $(ii)\Rightarrow(i)$ is trivial; and $(i)\Rightarrow(ii)$ follows from the fact that, since $D\seq C\seq B$, we have $d(a,D)\geq d(a,C)\geq d(a,B)$.

\emph{Extension}: It follows from Fact \ref{thm:UN}$(a)$ that, for all $A,B,C\subset\U_\cN$, there is $A'\equiv_C A$ such that, for all $a'\in A'$ and $b\in B$, $d(a',b)=d_{\max}(a',b/C)$. Note also that $d(a,C)\leq d_{\max}(a,b/C)$ for any $a,b\in\U_\cN$ and $C\subset\U_\cN$.

Now assume $A\ind^\infty_C B$ and $\hat{B}\supseteq B$. Let $A'\equiv_{BC} A$ such that, for all $a'\in A'$ and $b\in\hat{B}$, $d(a',b)=d_{\max}(a',b/BC)$. We show $A'\ind^\infty_C \hat{B}$. Fix $a'\in A'$ and suppose $d(a',\hat{B})=0$. Then $a'\in\hat{B}$, and so $d_{\max}(a',a'/BC)=d(a',a')=0$. Therefore $d(a',BC)=0$, which implies $d(a,BC)=0$. Then $d(a,C)=0$, and so $d(a',C)=0$. Next, fix $a'\in A'$ and suppose $d(a',C)=\infty$. Then $d(a,C)=\infty$, and so $A\ind^\infty_C B$ implies $d(a,BC)=\infty$. For any $b\in\hat{B}$, we have $d(a',b)=d_{\max}(a',b/BC)\geq d(a',BC)=d(a,BC)=\infty$, and so $d(a',\hat{B})=\infty$. 

\emph{Local character}: Fix $A,B\subset\U_\cN$. We will find $C\seq B$ such that $|C|\leq|A|$ and $A\ind^\infty_C B$. We may assume $B$ is nonempty. Given $a\in A$, choose $b_a\in B$ such that $d(a,B)=d(a,b_a)$. Let $C=\{b_a:a\in A\}$. For any $a\in A$, we have $d(a,C)=d(a,B)$, and this clearly implies $A\ind^\infty_C B$.   

\emph{Anti-reflexivity}: Trivial, since $\ind^\infty$ implies $\ind^a$. 
\end{proof}

\begin{corollary}\label{cor:TN}
$T^{\eq}_\cN$ has more than one strict independence relation.
\end{corollary}
\begin{proof}
We have shown that $\ind^a$ and $\ind^\infty$ are strict independence relations for $T_\cN$. Therefore, by Theorem \ref{thm:transfer} and Fact \ref{thm:UN}$(b)$, it suffices to show that these two relations are not the same in $\U_\cN$. To see this, fix distinct $a,b\in\U_\cN$ such that $d(a,b)<\infty$. Fix any subset $C\subset\U_\cN$ such that $d(a,C)=\infty=d(b,C)$ (e.g. $C=\emptyset$). Then $a\ind^a_C b$ by Fact \ref{thm:UN}$(a)$, and we clearly have $a\nind^\infty_C b$.
\end{proof}  

\section{Remarks and Generalizations}

A major motivation for Question 1.7 of \cite{Adgeo} comes from the following open problem in the study of simple theories (asked by several authors, e.g. \cite{Adgeo}, \cite{EaOn}). 

\begin{question}\label{Q:big}
Suppose $T$ is a simple theory. Are forking and thorn-forking the same in $T^{\eq}$?
\end{question}

This question is known to have a positive answer if $T$ is additionally assumed to eliminate hyperimaginaries or satisfy the stable forking conjecture (see \cite{EaOn}). If $T$ is simple then forking independence is the strongest strict independence relation for $T^{\eq}$ and thorn-forking independence is the weakest \cite{Adgeo}. So a negative answer to Question \ref{Q:big} would follow from the existence of a \emph{simple} theory $T$, such that $T^{\eq}$ has more than one strict independence relation. Therefore, it is worth observing that our example, $T_\cN$, is not simple. This is shown in \cite{CoDM2} (see Fact \ref{fact:SOPn} below), but also follows by adapting the proof in \cite{CaWa} that, for a fixed $n\geq 3$, the \Fraisse\ limit of metric spaces with distances in $\{0,1,\ldots,n\}$ is not simple. On the other hand, when $n=2$, this \Fraisse\ limit is precisely the countable random graph (where the graph relation is $d(x,y)=1$), which is well-known to have a simple theory. Therefore, in this context of ``generalized" Urysohn spaces, a natural question is whether there is a suitable choice of distance set such that, if $T$ is the theory of the associated Urysohn space, then $T$ is simple and the analog of $\ind^\infty$ still yields a strict independence relation for $T^{\eq}$ distinct from thorn-forking. The goal of this section is to demonstrate that this is unlikely. First, we define a precise context for studying Urysohn spaces over arbitrary distance sets. 

\begin{definition}\label{def:DM}
A structure $\cR=(R,\p,\leq,0)$ is a \textbf{distance monoid} if $(R,\p,0)$ is a commutative monoid and $\leq$ is a total, $\p$-invariant order with least element $0$. 
\end{definition}

Fix a countable distance monoid $\cR$. We define an \emph{$\cR$-metric space} to be a set equipped with an $\cR$-valued metric. Let $\cK_\cR$ denote the class of finite $\cR$-metric spaces. We consider $\cK_\cR$ as a class of $\mathcal{L}_R$-structures, where $\mathcal{L}_R=\{d_r(x,y):r\in R\}$ and $d_r(x,y)$ is a binary relation interpreted as $d(x,y)\leq r$. Then $\cK_\cR$ is a \Fraisse\ class by a similar argument as for $\cK_\cN$, or by adapting the following fact, due to Sauer \cite{Sa13b}, which also motivates Definition \ref{def:DM}.

\begin{fact}\label{fact:Sa}
Suppose $S\seq\R^{\geq0}$ is countable, contains $0$, and is closed under the operation $u+_S v:=\sup\{x\in S:x\leq u+v\}$. Then the class of finite metric spaces, with distances in $S$, is a \Fraisse\ class if and only if $+_S$ is associative on $S$ (and so $(S,+_S,\leq,0)$ is a distance monoid).
\end{fact}

\begin{definition}
Given a countable distance monoid $\cR$, we define the \textbf{$\cR$-Urysohn space}, denoted $\cU_\cR$, to be the \Fraisse\ limit of $\cK_\cR$. Let $T_\cR=\Th(\cU_\cR)$, and let $\U_\cR$ denote a sufficiently saturated monster model of $T_\cR$.  
\end{definition}

When working with $T_\cR$, we again face the obstacle that $\U_\cR$ cannot be interpreted as an $\cR$-metric space. As in the last section, this is resolved by constructing a distance monoid extension $\cR^*=(R^*,\p,\leq,0)$ of $\cR$. The construction is quite technical, and so we refer the reader to \cite{CoDM2}. The rough idea is that the set $R^*$ corresponds to the space of quantifier-free $2$-types (over $\emptyset$) consistent with $T_\cR$, and the ordering $\leq$ and operation $\p$ extend to $R^*$ in a canonical way. Then, given a model $M\models T_\cR$, one defines an $\cR^*$-metric on $M$ by setting the distance between two points $a,b\in M$ to be the element of $R^*$ corresponding the the quantifier-free $2$-type realized by $(a,b)$.

For example, define $\cQ:=(\Q^{\geq0},+,\leq,0)$. Then in $\cQ^*$, the set $(\Q^{\geq0})^*$ can be identified with $\R^{\geq0}\cup\{q^+:q\in\Q^{\geq0}\}\cup\{\infty\}$, where $q^+$ is a new symbol for an immediate successor of $q$ and $\infty$ is a new symbol for an infinite element. 

Suppose $\cR$ is a countable distance monoid. The goal of this section is to show that if $\ind^\infty$ can be defined for $T_\cR$, and moreover yields a strict independence relation for $T^{\eq}_\cR$ distinct from thorn-forking, then $T_\cR$ is not simple. The next definition lists the properties we need to define $\ind^\infty$ on $T_\cR$, prove $\ind^\infty$ is a strict independence relation, and lift $\ind^\infty$ to $T^{\eq}_\cR$ using Theorem \ref{thm:transfer}.

\begin{definition}\label{def:suit}
A countable distance monoid $\cR$ is \textbf{suitable} if $\cR$ has no maximal element, $T_\cR$ has quantifier elimination, and $T_\cR$ has weak elimination of imaginaries.
\end{definition}

The following fact provides natural examples of suitable distance monoids.

\begin{fact}\textnormal{\cite{CoDM2}}
If $(G,\p,\leq,0)$ is a nontrivial countable ordered abelian group, and $\cR=(G^{\geq0},\p,\leq,0)$, then $\cR$ is a suitable distance monoid.
\end{fact}

Suppose $\cR$ is a suitable distance monoid. By quantifier elimination, we again have $\acl(C)=C$ for all $C\subset\U_\cR$, and that $\U_\cR$ is a $\kappa^+$-universal and $\kappa$-homogeneous $\cR^*$-metric space, where $\kappa$ is the saturation cardinal of $\U_\cR$. Moreover, since $\cR$ has no maximal element, the type $\{\neg d_r(x,y):r\in R\}$ is realized in $\U_\cR$, and therefore corresponds to a new element $\infty$ in $\cR^*$. So we may define $\ind^\infty$ on $\U_\cR$ exactly as in Definition \ref{def:II} (with $\cR$ in place of $\cN$). 

\begin{theorem}
If $\cR$ is a suitable distance monoid then $\ind^\infty$ is a strict independence relation for $T_\cR$, which is distinct from $\ind^{\thrn}$.
\end{theorem}
\begin{proof}
First, we have  $\ind^a=\ind^{\thrn}$ in $\U_\cR$ by Fact \ref{fact:alg}. Next, it follows from the construction of $\cR^*$ in \cite{CoDM2} that $0$ always has an immediate successor in $\cR^*$ (e.g. $0^+$ in $\cQ^*$). Therefore, given $a\in\U_\cR$ and $C\subset\U_\cR$, we have $d(a,C)=0$ if and only if $a\in C$. In particular, $\ind^\infty$ implies $\ind^a$, and the same argument as in the proof of Corollary \ref{cor:TN} shows that $\ind^\infty$ and $\ind^a$ are distinct. Finally, the argument that $\ind^\infty$ is a strict independence relation for $T_\cR$ is the same as Theorem \ref{thm:inftySIR} for $T_\cN$, except for the following issues.

To show \emph{symmetry}, we need the following properties of $\cR^*$, which follow from \cite[Proposition 2.11]{CoDM2}. First, if $a,b\in\U_\cR$ and $C\subset\U_\cR$ then $d(b,C)\leq d(a,b)\p d(a,C)$. Second, if $r,s\in R^*$ then $r\p s=\infty$ if and only if $r=\infty$ or $s=\infty$. Given this, the proof of symmetry follows as in Theorem \ref{thm:inftySIR}.  

The argument for \emph{local character} requires slightly more work. Fix $A,B\subset\U_\cR$. Given $a\in A$, define $X_a\seq R^*$ such that $r\in X_a$ if and only if there is $b\in B$ with $d(a,b)=r$. For each $a\in A$ and $r\in X_a$, fix some $b^a_r\in B$ such that $d(a,b^a_r)=r$. Now let $C=\{b^a_r:a\in A,~r\in X_a\}$. Then $|C|\leq |A|+|R^*|$ and $d(a,B)=d(a,C)$ for any $a\in A$, which implies $A\ind^\infty_C B$. 

To show \emph{extension}, we must define the analog of $d_{\max}$. In particular, given $C\subset\U_\cR$ and $a,b\in\U_\cR$, set $d_{\max}(a,b/C):=\inf_{c\in C}(d(a,c)\p d(c,b))$. This definition is justified by the fact that the ordering in $\cR^*$ is complete (again, see \cite{CoDM2}). The proof of extension now follows as in Theorem \ref{thm:inftySIR}.
\end{proof}

Applying Theorem \ref{thm:transfer}, we have:

\begin{corollary}
If $\cR$ is a suitable distance monoid, then $T^{\eq}_\cR$ has more than one strict independence relation.
\end{corollary}

To finish our goal, we quote  \cite{CoDM2} to show that if $\cR$ is suitable then $T_\cR$ is not simple. In fact, $T_\cR$ is quite complicated in the sense of certain combinatorial dividing lines invented by Shelah (see \cite{CoDM2} or \cite{Sh500} for definitions).

\begin{fact}\label{fact:SOPn}\textnormal{\cite{CoDM2}}
If $\cR$ is a suitable distance monoid then $T_\cR$ has $\textnormal{SOP}_n$ for all $n\geq 3$ (but does not have the strict order property). Therefore $T_\cR$ is not simple.
\end{fact}
\begin{proof}
We fix $n\geq 3$ and explain how to obtain $\text{SOP}_n$ from results in \cite{CoDM2}. Fix $r\in R^{>0}$. We show $(n-1)r<nr$, which implies $\text{SOP}_n$ by a technical construction in \cite[Section 6]{CoDM2}. If $(n-1)r=nr$ then $(n-1)r=mr$ for all $m\geq n-1$, and so $d(x,y)\leq (n-1)r$ is a nontrivial equivalence relation on $\U_\cR$. This is shown in \cite[Section 7]{CoDM2} to contradict weak elimination of imaginaries.\end{proof}

A final observation is that this is not the first time nonstandard distances in saturated models of metric spaces have led to interesting model theoretic phenomena. In particular, consider the \emph{rational Urysohn space} $\cU_\cQ$. In the monster model $\U_\cQ$, we have the type-definable binary relations $d_{0^+}(x,y):=\bigwedge_{r\in \Q^+}d_r(x,y)$ and $d_\infty(x,y):=\bigwedge_{r\in\Q^+}\neg d_r(x,y)$ describing, respectively, infinitesimal distance and infinite distance. Note that $d_{0^+}(x,y)$ and the complement of $d_\infty(x,y)$ are both equivalence relations on $\U_\cQ$. The work in this section uses $d_\infty(x,y)$ to obtain a strict independence relation for $T^{\eq}_\cQ$ distinct from thorn-forking. In \cite{CaWa}, Casanovas and Wagner used $d_{0^+}(x,y)$ to obtain the first example of non-eliminable (finitary) hyperimaginaries in a theory without the strict order property (to be precise, they considered the rational Urysohn \emph{sphere}; see \cite[Section 7]{CoDM2} for details). In both of these situations, the aberrant behavior can be traced to a failure of $\aleph_0$-categoricity. Specifically, as types in $S_2(T_\cQ)$, $d_{0^+}(x,y)$ and $d_\infty(x,y)$ are both non-isolated. Moreover, it is a fact that finitary hyperimaginaries are always eliminated in $\aleph_0$-categorical theories \cite[Theorem 18.14]{Cabook}.  So we ask the following question.

\begin{question}
Suppose $T$ is a complete $\aleph_0$-categorical theory. Can $T^{\eq}$ have more than one strict independence relation?
\end{question}

\subsection*{Acknowledgements}
I would like to thank the anonymous referee for several corrections and suggestions, which significantly improved the final draft.

\bibliography{/Users/gabrielconant/Desktop/Math/BibTex/biblio}
\bibliographystyle{amsplain}
\end{document}